\documentclass[a4paper,12pt]{article}
\usepackage{geometry,latexsym,amssymb,amsmath,amsthm,color,bm}
\usepackage[latin5]{inputenc}
\usepackage{enumerate}
\usepackage[T1]{fontenc}
\usepackage{authblk}
\usepackage[all]{xy}
\usepackage{enumitem}
\usepackage{palatino}
\usepackage{indentfirst}
\usepackage{titlesec}
\usepackage{setspace}
\newtheorem{theorem}{Theorem}[section]
\newtheorem{proposition}{Proposition}[section]

\newtheorem{remark}{Remark}[section]
\newtheorem{corollary}{Corollary}[section]
\newtheorem{definition}{Definition}[section]
\newtheorem{example}{Example}[section]

\def\w{\widetilde}
\def\Ob{\operatorname{Ob}}
\def\St{\operatorname{St}}

\def\Aut{\operatorname{Aut}}
\def\inc{\operatorname{inc}}
\def\Der{\operatorname{Der}}
\def\D{\operatorname{D}}

\def\CXM{\bm{\mathsf{CovXMod/(A,B,\alpha)}}}
\def\GGdA(G){\bm{\mathsf{GpGpdAct(G)}}}
\def\PGGdA(G){\bm{\mathsf{PGpGpdAct(G)}}}
\def\GGdAct(GT){\bm{\mathsf{GpGpdAct(\w{G})}}}
\def\PGGdAct(GT){\bm{\mathsf{PGpGpdAct(\w{G})}}}
\def\GGdC/G{\bm{\mathsf{GpGpdCov/G}}}
\def\GGdCov/X{\bm{\mathsf{GpGpdCov/\pi X}}}
\def\GdC/G{\bm{\mathsf{GpdCov/G}}}
\def\TGrC/X{\mathsf{TGrCov/X}}
\def\GdA(G){\bm{\mathsf{GpdAct(G)}}}
\def\Act(G){\bm{\mathsf{GpdAct(G)}}}
\def\Cov/G{\bm{\mathsf{GpdCov/G}}}
\def\C{\bm{\mathsf{C}}}
\def\LXM{\bm{\mathsf{LXMod/}}}

\def\XMod{\bm{\mathsf{XMod}}}
\def\GpGd{\bm{\mathsf{GpGpd}}}

\begin{document}
\title{Further remarks on liftings of crossed modules}

\author[]{Tunçar Şahan}
\affil[]{\small{Department of Mathematics, Aksaray University, Aksaray, TURKEY \\ \textit{tuncarsahan@gmail.com}}}

\date{}

\maketitle
\begin{abstract}
In this paper we define the notion of pullback lifting of a lifting crossed module over a crossed module morphism and interpret this notion in the category of group-groupoid actions as pullback action. Moreover, we give a criterion for the lifting of homotopic crossed module morphisms to be homotopic, which will be called homotopy lifting property for crossed module morphisms. Finally, we investigate some properties of derivations of lifting crossed modules according to base crossed module derivations.
\end{abstract}

\\ {\bf Classification:} 20L05,  57M10,  22AXX, 22A22, 18D35

\section{Introduction}

Covering groupoids have an important role in the applications of groupoids (see for example  \cite{Br1} and \cite{Hi}). It is well known that for a groupoid $G$, the category $\Act(G)$ of  groupoid actions  of $G$ on  sets, these are also called operations or $G$-sets, are equivalent to  the category $\Cov/G$ of  covering groupoids of $G$. In \cite[Theorem 2]{Br-Da-Ha} analogous result has been proved for topological groupoids.

Group-groupoids are internal categories, and hence internal groupoids in the category of groups, or equivalently group objects in the category of small categories \cite{BS1}. If $G$ is a group-groupoid, then the category $\GGdC/G$ of group-groupoid coverings of $G$ is equivalent to the category $\GGdA(G)$ of group-groupoid actions of $G$ on groups \cite[Proposition 3.1]{Br-Mu1}. In \cite{Ak-Al-Mu-Sa} this result has recently generalized to the case where $G$ is an internal groupoid for an algebraic category $\C$, acting on a group with operations. See \cite{Mu-Sa} for covering groupoids of categorical groups, \cite{Br-Mu1,Mu1} for coverings of crossed modules of groups and \cite{Mu-Tu} for coverings and crossed modules of topological groups with operations.

Whitehead introduced the notion of crossed module, in a series of papers as algebraic models for (connected) homotopy 2-types (i.e. connected spaces with no homotopy group in degrees above 2), in much the same way that groups are algebraic models for homotopy 1-types \cite{Wth1,Wth3,Wth2}. A crossed module consists of groups $A$ and $B$, where $B$ acts on $A$, and a homomorphism of groups $\alpha\colon A\rightarrow B$ satisfying $\alpha(b\cdot a)=b+\alpha(a)-b$ and $ {\alpha(a)}\cdot a_1 = a+a_1-a$ for all $a,a_1 \in A$ and $b\in B$. Crossed modules can be viewed as 2-dimensional groups \cite{BrLowDim} and have been widely used in: homotopy theory, \cite{BHS}; the theory of identities among relations for group presentations, \cite{BrownHubesshman}; algebraic K-theory \cite{Loday}; and homological algebra, \cite{Hub,Lue}.

In \cite{BS1} it was proved that the categories of crossed modules and group-groupoids, under the name of $\mathcal{G}$-groupoids, are equivalent (see also  \cite{Loday82} for an alternative equivalence in terms of an algebraic object called {\em cat$^n$-groups}). By applying this equivalence of the categories,  normal and quotient objects in the category of group-groupoids  have been recently obtained in \cite{Mu-Sa-Al}. Moreover in \cite{Mu-Tu-Ar} authors have interpreted in the category of crossed modules the notion of action of a group-groupoid on a group via a group homomorphism and hence introduced the notion of lifting of a crossed module. Further they showed some results on liftings of crossed modules and proved that the category $\LXM{\bm{\mathsf{(A,B,\alpha)}}}$ of liftings of crossed modules, the category $\CXM$ of covering crossed modules and the category $\GGdA(G)$ of group-groupoid actions are equivalent.

In this study, first of all, we give pulback of liftings of crossed modules. Further, we obtain a property which will be called homotopy lifting property for crossed modules. Finally, we lift the derivations of base crossed module to lifting crossed modules.

\section{Preliminaries}

Let  $G$  be a groupoid.  We write  $\Ob(G)$  for the set of
objects  of    $G $ and write $G$ for the set of morphisms. We also identify  $\Ob(G)$  with the set of
identities of  $G $ and so an  element  of  $\Ob(G)$  may be written as
$x$  or  $1_x$  as convenient.  We  write   $d_0, d_1 \colon
G\rightarrow \Ob(G)$  for the source and target maps, and, as usual,
write $G(x,y)$ for $d_0^{-1}(x)\cap d_1 ^{-1}(y)$, for $x,y\in \Ob(G)$.
The composition  $h\circ g$  of two elements of  $G$  is defined if
and only if  $d_0(h) =d_1(g)$, and so the   map $(h,g)\mapsto h\circ g$
is defined on the pullback  $G {_{d_0}\times_{d_1}} G$
of $d_0$  and $d_1 $.  The \emph{inverse} of $g\in G(x,y)$ is denoted by
$g^{-1}\in G(y,x)$. If  $x\in \Ob(G) $, we write $\St_Gx$  for $d_0^{-1}(x) $  and  call
the \emph{star} of $G$ at $x$. The set of all
morphisms from $x$ to $x$ is a group, called \emph{object group}
at $x$, and denoted by $G(x)$.

A groupoid $G$ is \emph{transitive (resp. simply transitive, 1-transitive or totally intransitive)} if
$G(x,y)\neq\emptyset$ (resp. $G(x,y)$ has no more than one element, $G(x,y)$ has exactly one element or
$G(x,y)=\emptyset$) for all $x,y\in \Ob(G)$ such that $x\neq y$.

Let $p\colon\widetilde G\rightarrow G$ be a morphism of groupoids. Then $p$ is
called a \emph{covering morphism} and $\widetilde{G}$  a \emph{covering groupoid} of $G$ if for
each $\widetilde x\in \Ob(\widetilde G)$ the restriction  $\St_{\widetilde{G}}{\widetilde{x}} \rightarrow \St_{G}{p(\widetilde x)}$  is  bijective.

Assume that $p\colon \w{G}\rightarrow G$ is a covering morphism. Then  we have a lifting function $S_{p}\colon G_{s}\times_{\Ob(p)}\Ob(\w{G})\rightarrow \w{G}$ assigning to the pair $(a,x)$ in the pullback $G_{s}\times_{\Ob(p)}\Ob(\w{G})$ the unique element $b$ of $\St_{\widetilde{G}}{x}$ such that $p(b)=a$. Clearly $S_{p}$ is inverse to $(p,s)\colon  \widetilde{G}\rightarrow G_{s}\times_{\Ob(p)}\Ob(\widetilde{G})$. So it is stated that $p\colon  \widetilde{G}\rightarrow G$ is a covering morphism if and only if $(p,s)$ is a  bijection.

A covering morphism $p\colon \widetilde{G}\rightarrow
G$ is called \emph{transitive } if both $\widetilde{G}$ and $G$ are
transitive.  A transitive covering morphism $p\colon\widetilde G\rightarrow G$
is called \emph{universal} if $\widetilde G$ covers every cover of $G$, i.e.,  for every covering morphism $q\colon \widetilde{H}\rightarrow G$ there is a unique
morphism of groupoids $\widetilde{p}\colon \widetilde G\rightarrow \widetilde{H}$ such that $q\widetilde{p}=p$
(and hence $\widetilde{p}$ is also a covering morphism), this is equivalent to that for
$\widetilde{x}, \widetilde{y}\in \Ob({\widetilde G})$ the set $\widetilde{G}(\widetilde x, \widetilde y)$
has not more than one element.

Recall that an action of a groupoid $G$ on a set $S$ via a function $\omega\colon S\rightarrow \Ob(G)$ is a function ${G}_{d_{0}}\times_\omega S\rightarrow S, (g,x)\mapsto g\bullet x$ satisfying the usual rule for an action, namely $\omega(g\bullet x)=d_1(g)$, $1_{\omega(s)}\bullet s=s$ and $(h\circ g)\bullet s=h\bullet (g\bullet s)$ whenever defined. A morphism $f\colon (S,\omega)\rightarrow (S',\omega')$ of such actions is a function $f\colon S\rightarrow S'$ such that $w'f=w$ and $f(g\bullet s)=g\bullet f(s)$ whenever $g\bullet s$ is defined. This gives a category $\Act(G)$ of actions of $G$ on sets. For such an action the action groupoid $G\ltimes S$ is defined to have object set $S$, morphisms the pairs $(g,s)$ such that $d_0(g)=\omega(s)$, source and target maps $d_0(g,s)=s$, $d_1(g,s)=g\bullet s$, and the composition
$(g',s')\circ (g,s)=(g\circ g',s)$
whenever $s'=g\bullet s$. The projection $q\colon G\ltimes S\rightarrow G, (g,s)\mapsto s$ is a covering morphism of groupoids and the functor assigning this covering morphism to an action gives an equivalence of the categories $\Act(G)$ and $\Cov/G$.

Let $G$ be a group-groupoid. An action of the group-groupoid $G$ on a group $X$ via $\omega$ consists of a morphism $\omega\colon X\rightarrow \Ob(G)$ from the group $X$ to the underlying group of $\Ob(G)$ and an action of the groupoid $G$ on the underlying set $X$  via $\omega$ such that the following interchange law holds:
\begin{equation*}
(g\bullet x)+(g'\bullet x')=(g+ g')\bullet(x+x') \label{interchangeaction}
\end{equation*}
whenever both sides are defined. A morphism $f\colon (X,\omega)\rightarrow (X',\omega')$ of such actions is a morphism   $f\colon X\rightarrow X'$ of groups and of the underlying operations of $G$.  This gives a category $\GGdA(G)$ of actions of $G$ on groups. For an action of $G$ on the group $X$ via $\omega$, the action groupoid $G\ltimes X$ has a group structure defined by \[(g,x)+(g',x')=(g+g',x+x')\]
and with this operation $G\ltimes X$ becomes a group-groupoid and the projection $p\colon G\ltimes X\rightarrow G$ is an object of the category $\GGdC/G$. By means of this construction the following equivalence of the categories was given in \cite{Br-Mu1}.

\begin{proposition}\cite[Proposition 3.1]{Br-Mu1} The categories $\GGdC/G$ and $\GGdA(G)$ are equivalent.\end{proposition}

We recall that a crossed module of groups consists of two groups $A$ and $B$,
an action of $B$ on $A$ denoted by $b\cdot a$ for $a\in A$ and $b\in B$;
and a morphism  $\alpha\colon A\rightarrow B$ of groups satisfying the
following conditions for all $a,a_1\in A$ and $b\in B$
\begin{enumerate}[label=\textbf{CM\arabic{*})}, leftmargin=2cm]
	\item\label{CM1} $\alpha(b\cdot a)=b+\alpha(a)-b$,
	\item\label{CM2} $\alpha(a)\cdot a_1=a+a_1-a$.
\end{enumerate}
We will denote such a crossed module by $(A,B,\alpha)$.

For instance, the inclusion map of a normal subgroup is a crossed module with the conjugation action. As a motivating geometric example of crossed modules due to Whitehead \cite{Wth1,Wth2} if $X$ is  topological space and $A\subseteq X$ with $x\in A$, then  there is a natural action of $\pi_1(A,x)$ on second relative homotopy group $\pi_2(X,A,x)$ and with this action the boundary map \[\partial\colon\pi_2(X,A,x)\rightarrow \pi_1(A,x)\] becomes a crossed module. This crossed module is called \emph{fundamental crossed module} and denoted by $\Pi(X,A,x)$ (see \cite{BH} for further information).

The following are some standard properties of crossed modules.

\begin{proposition}
	Let $(A,B,\alpha)$ be a crossed module. Then
	\begin{enumerate}[label=(\roman{*}), leftmargin=1cm]
		\item  $\alpha(A)$ is a normal subgroup of $B$.
		\item $\ker \alpha$ is central in $A$, i.e. $\ker \alpha$ is a subset of $Z(A)$, the center of $A$.
		\item $\alpha(A)$ acts trivially on $Z(A)$.
	\end{enumerate}
\end{proposition}

Let $(A,B,\alpha)$ and $(A',B',\alpha')$ be two crossed modules. A morphism  $(f_1,f_2)$
from $(A,B,\alpha)$ to $(A',B',\alpha')$  is a pair of morphisms of groups
$f_1\colon A\rightarrow A'$ and $f_2\colon B\rightarrow B'$  such that
$f_2\alpha=\alpha'f_1$ and  $f_1(b\cdot a)=f_2(b)\cdot f_1(a)$ for  $a\in A$ and $b\in B$.

Crossed modules with morphisms between them form a category which is called the category of crossed modules and denoted by $\XMod$. It was proved by Brown and Spencer in \cite[Theorem 1]{BS1} that the category $\XMod$ of crossed modules over  groups is equivalent to the category $\GpGd$ of group-groupoids. In \cite{Mu-Tu-Ar} certain groupoid properties adapted for crossed modules as in the following definition.

\begin{definition}\cite{Mu-Tu-Ar}
	Let $(A,B,\alpha)$ be a crossed module. Then $(A,B,\alpha)$ is called \emph{transitive (resp. simply transitive, 1-transitive or totally intransitive)} if $\alpha$ is surjective (resp. injective, bijective or zero morphism and $A$ is abelian).
\end{definition}

\begin{example} If $X$ is a topological group whose underlying topology is path-connected (resp. totally disconnected), then the crossed module $(\St_{\pi X}0,X,t)$ is transitive (resp. totally intransitive).
\end{example}

\begin{example}
	$(N,G,\inc)$ is a simply transitive crossed module.
\end{example}

The notion of litfing of a crossed module was introduced in \cite{Mu-Tu-Ar} as an interpretation of the notion of action group-groupoid in the category of crossed modules over groups. Now we will recall the definition of lifting of a crossed module and some examples. Following that we will give further results on liftings of crossed modules.

\begin{definition}\cite{Mu-Tu-Ar} Let $(A,B,\alpha)$ be a crossed module and $\omega\colon X\rightarrow B$ a morphism of groups.  Then a crossed module $(A,X,\varphi)$, where the action of $X$ on $A$ is defined via $\omega$, is a called a \emph{lifting of $\alpha$ over $\omega$} and  denoted by  $(\varphi,X,\omega)$ if the following diagram is commutative.
	\[\xymatrix{ & X \ar[d]^\omega\\
		A \ar[r]_-\alpha \ar@{-->}[ur]^\varphi & B}\]
\end{definition}

\begin{remark}\label{remlift}
	Let $(\varphi,X,\omega)$ be a lifting of $(A,B,\alpha)$ then $\ker\varphi\subseteq\ker\alpha$ and $(1_A,\omega)$ is a morphism of crossed modules.
\end{remark}

Here are some examples of liftings of crossed modules:
\begin{enumerate}[label=(\roman{*}), leftmargin=2cm]
	\item Let $G$ be a group with trivial center. Then the automorphism crossed module $G\rightarrow \Aut(G)$ is simply transitive and hence every lifting of $G\rightarrow \Aut(G)$ is simply transitive.
	
	\item \label{topliftexam}
	If $p\colon \w{X}\rightarrow X$ is a covering morphism of topological groups. Then with the morphism $\w{t}\colon \St_{\pi X}0\rightarrow \w{X}$, $\w{t}([\alpha])\mapsto\w{\alpha}(1)$, the final point of the lifted path of $\alpha$ at ${\w{0}}$, $(\w{t},\w{X},p)$ becomes a lifting of $(\St_{\pi X}0 ,X,t)$.
	
	\item \label{autxmod}
	Every crossed module $(A,B,\alpha)$ is a lifting of the automorphism crossed module $(A,\Aut(A),\iota)$ over the action of $B$ on $A$, i.e. $\theta\colon B\rightarrow \Aut(A)$, $\theta(b)(a)=b\cdot a$.
	
	\item \label{natlift}
	Every crossed module $(A,B,\alpha)$ lifts to the crossed module $(A,A/N,p)$ over $\omega\colon A/N\rightarrow B$, $a+N\mapsto \alpha(a)$ where $N=\ker \alpha$.
\end{enumerate}

\section{Results on Liftings of crossed modules}

\subsection{Pullback liftings}

Mucuk and \c{S}ahan \cite{Mu-Tu-Ar} gave the following theorem as a criterion for the existence of the lifting crossed module.

\begin{theorem}\cite[Theorem 4.5]{Mu-Tu-Ar}\label{existlift}
	Let $(A,B,\alpha)$ be a crossed module and $C$ be a subgroup of $\ker \alpha$. Then there exist a lifting $(\varphi,X,\omega)$ of $\alpha$ such that $\ker \varphi=C$. Moreover in this case $\ker\omega=\ker \alpha / C$.
\end{theorem}

Every crossed module $(A,B,\alpha)$ lifts to itself over the identity morphism $1_B$ on $B$. Let $(\varphi,X,\omega)$ and $(\varphi',X',\omega')$ be two liftings of $(A,B,\alpha)$. A morphism $f$ from $(\varphi,X,\omega)$ to $(\varphi',X,\omega')$ is a group homomorphism $f\colon X\rightarrow X'$ such that $f\varphi=\varphi'$ and $\omega'f=\omega$. The category of liftings of a crossed module $(A,B,\alpha)$ is denoted by $\LXM{\bm{\mathsf{(A,B,\alpha)}}}$.

Now we will define the pullback of a lifting crossed module over a crossed module morphism with codomain the base crossed module.

\begin{proposition}\label{pulllift}
	Let $(A,B,\alpha)$, $(\w{A},\w{B},\w{\alpha})$ be two crossed modules, $(\varphi,X,\omega)$ a lifting of $(A,B,\alpha)$ and $(f,g)\colon(\w{A},\w{B},\w{\alpha})\rightarrow (A,B,\alpha)$ a crossed module morphism. Then $(\psi,X{_{\omega}\times_{g}}\w{B},\pi_2)$ is a lifting of $(\w{A},\w{B},\w{\alpha})$, where $\psi=(\varphi f,\w{\alpha})$ and $X{_{\omega}\times_{g}}\w{B}$ is the pullback group of $X$ and $\w{B}$. Moreover $(f,\pi_1)\colon(\w{A},X{_{\omega}\times_{g}}\w{B},\psi)\rightarrow (A,X,\varphi)$ is a crossed module morphism.
\end{proposition}

\begin{proof}
	By the definition of $\psi$, it is easy to see that $\pi_2\psi=\w{\alpha}$. Hence, we only need to show that $(\w{A},X{_{\omega}\times_{g}}\w{B},\psi)$ is a crossed module, where the action of $X{_{\omega}\times_{g}}\w{B}$ on $\w{A}$ given by $(x,\w{b})\cdot \w{a}=\w{b}\cdot \w{a}$.
	\begin{enumerate}[label=\textbf{CM\arabic{*})}, leftmargin=2cm]
		\item\label{CM1} Let $(x,\w{b})\in X{_{\omega}\times_{g}}\w{B}$ and $\w{a}\in \w{A}$. Then,
		\begin{align*}
		\psi((x,\w{b})\cdot \w{a}) &= \psi(\w{b}\cdot \w{a}) \\
		&= (\varphi f,\w{\alpha})(\w{b}\cdot \w{a}) \\
		&= (\varphi f(\w{b}\cdot \w{a}),\w{\alpha}(\w{b}\cdot \w{a})) \\
		&= (\varphi(g(\w{b})\cdot f(\w{a})) ,\w{b}+\w{\alpha}(\w{a})-\w{b}) \tag{since $g(\w{b})=\omega(x)$} \\
		&= (\varphi(\omega(x)\cdot f(\w{a})) ,\w{b}+\w{\alpha}(\w{a})-\w{b}) \\
		&= (\varphi(x\cdot f(\w{a})) ,\w{b}+\w{\alpha}(\w{a})-\w{b}) \\
		&= (x+\varphi(f(\w{a}))-x ,\w{b}+\w{\alpha}(\w{a})-\w{b}) \\
		&= (x,\w{b})+(\varphi(f(\w{a})),\w{\alpha}(\w{a}))-(x,\w{b}) \\
		&= (x,\w{b})+\psi(\w{a})-(x,\w{b}).
		\end{align*}
		\item\label{CM2} Let $\w{a},\w{a_1}\in \w{A}$ and $\w{a}\in \w{A}$. Then,
		\begin{align*}
		\psi(\w{a_1})\cdot \w{a} &= ((\varphi f,\w{\alpha})(\w{a_1}))\cdot \w{a} \\
		&= (\varphi f(\w{a_1}),\w{\alpha}(\w{a_1}))\cdot \w{a} \\
		&= \w{\alpha}(\w{a_1})\cdot \w{a} \\
		&= \w{a_1}+\w{a}-\w{a_1}.
		\end{align*}
	\end{enumerate}
	Thus $(\w{A},X{_{\omega}\times_{g}}\w{B},\psi)$ is a crossed module, i.e., $(\psi,X{_{\omega}\times_{g}}\w{B},\pi_2)$ is a lifting of $(\w{A},\w{B},\w{\alpha})$. Other details are straightforward.
\end{proof}

\[\xymatrix@!0@C=16mm@R=16mm{
	& & X \ar@{->}[dd]^-{\omega}  & \\	
	&     &  & X{_{\omega}\times_{g}}\w{B} \ar[dd]^{\pi_2} \ar[ul]_-{\pi_1}
	\\
	A \ar@{-->}[uurr]^{\varphi}\ar@{->}[rr]^-{\alpha}   & & B & \\	
	& \w{A} \ar@{..>}'[ur]|-{~~\psi=(\varphi f,\w{\alpha})~~}[uurr] \ar[ul]^-{f} \ar[rr]_{\w{\alpha}}   &  & \w{B} \ar[ul]_-{g}
}\qquad
\xymatrix@!0@C=16mm@R=16mm{
	& & X \ar@{->}[dd]^-{\omega} \ar@{->}[rr]^-{h}  & & X' \ar@{->}'[d][dd]^-{\omega'} \\	
	&     &  & X{_{\omega}\times_{g}}\w{B} \ar@{->}[rr]^(.4){h\times 1} \ar[dd]^(.7){\pi_2} \ar[ul]_-{\pi_1}
	& & X'{_{\omega'}\times_{g}}\w{B} \ar[dd]^-{\pi_2} \ar[ul]_-{\pi_1}\\
	A \ar@{-->}[uurr]^{\varphi}\ar@{->}[rr]^-{\alpha}   & & B   \ar@{=}'[r][rr] & & B  \\	
	& \w{A} \ar@{..>}'[ur]|-{~~\psi=(\varphi f,\w{\alpha})~~}[uurr] \ar[ul]^-{f} \ar[rr]_{\w{\alpha}}   &  & \w{B} \ar[ul]_-{g} \ar@{=}[rr] & & \w{B} \ar[ul]_-{g}
}\]

The lifting $(\psi,X{_{\omega}\times_{g}}\w{B},\pi_2)$ constructed in the previous proposition is called \textit{pullback lifting of} $(\varphi,X,\omega)$ \textit{over} $(f,g)$. If $(\varphi',X',\omega')$ is another lifting of $(A,B,\alpha)$ and $h\colon (\varphi,X,\omega)\rightarrow (\varphi',X',\omega')$ a morphism of liftings, then $(\psi',X'{_{\omega'}\times_{g}}\w{B},\pi_2)$ is also a lifting of $(\w{A},\w{B},\w{\alpha})$ and $h\times 1\colon (\psi,X{_{\omega}\times_{g}}\w{B},\pi_2)\rightarrow (\psi',X'{_{\omega'}\times_{g}}\w{B},\pi_2)$ is a morphism of liftings.

This construction is functorial, that is, defines a functor
\[\lambda^{\ast}_{(f,g)}\colon \LXM{\bm{\mathsf{(A,B,\alpha)}}} \rightarrow \LXM{\bm{\mathsf{(\w{A},\w{B},\w{\alpha})}}}\]
given by $\lambda^{\ast}_{(f,g)}((\varphi,X,\omega))=(\psi,X{_{\omega}\times_{g}}\w{B},\pi_2)$ on objects and by $\lambda^{\ast}_{(f,g)}(h)=h\times 1$ on morphisms, for each crossed module morphism $(f,g)\colon(\w{A},\w{B},\w{\alpha})\rightarrow (A,B,\alpha)$. Using the equivalence of the categories of liftings of a crossed module and of group-groupoid actions, we will define the corresponding notion in the category of group-groupoid actions.

\begin{proposition}
	Let $G$ and $\w{G}$ be two group-groupoids and $G$ acts on a group $X$ via a group homomorphism $\omega\colon X\rightarrow \Ob(G)$. If there is a group-groupoid morphism $f=(f_1,f_0)\colon \w{G}\rightarrow G$ then, $\w{G}$ acts on the pullback group $X{_{\omega}\times_{f_0}}\Ob(\w{G})$ via $\pi_2$.
\end{proposition}

\begin{proof}
	It is similar to the proof of Proposition \ref{pulllift}.
\end{proof}

Action defined above is called \textit{pullback action of} $(X,\omega)$ \textit{over} $f$. We also obtain a pullback functor as in the case of liftings:

\[\lambda^{\ast}_{f}\colon \GGdA(G) \rightarrow \GGdAct(GT)\]
given by $\lambda^{\ast}_{f}((X,\omega))=(X{_{\omega}\times_{f_0}}\Ob(\w{G}),\pi_2)$ on objects and by $\lambda^{\ast}_{f}(h)=h\times 1$ on morphisms, for each group-groupoid morphism $f\colon\w{G}\rightarrow G$.

\subsection{Homotopy lifting property}

Following theorem, given in \cite{Mu-Tu-Ar}, states the conditions for when a crossed module morphism lifts to a lifting of a crossed module.

\begin{theorem}\cite[Theorem 4.3]{Mu-Tu-Ar}\label{morplift}
	Let $(f,g)\colon(\widetilde{A},\widetilde{B},\widetilde{\alpha})\rightarrow(A,B,\alpha)$ be a morphism of crossed modules  where $(\widetilde{A},\widetilde{B},\widetilde{\alpha})$ is transitive and let $(\varphi,X,\omega)$ be a lifting of $(A,B,\alpha)$. Then there is a unique morphism of crossed modules $(f,\widetilde{g})\colon(\widetilde{A},\widetilde{B},\widetilde{\alpha})\rightarrow(A,X,\varphi)$ such that $\omega\widetilde{g}=g$ if and only if $f(\ker \widetilde{\alpha})\subseteq\ker \varphi$.
\end{theorem}

Note that, in the above theorem, since $(\widetilde{A},\widetilde{B},\widetilde{\alpha})$ is transitive then $\widetilde{\alpha}$ is surjective and there exists an $\widetilde{a}\in \widetilde{A}$ such that $\widetilde{\alpha}(\widetilde{a})=\widetilde{b}$ for all $\w{b}\in\w{B}$. So the morphism $\widetilde{g}$ is given by
\[\begin{array}{rcccl}
\w{g}&\colon&\w{B} & \longrightarrow & B \\
& & \w{b} & \longmapsto & \w{g}(\w{b})=\varphi f(\w{a})
\end{array}\label{gtilde}\]
for all $\w{b}\in\w{B}$.

According to Theorem \ref{morplift}, now we prove that liftings of homotopic crossed module morphisms as in Theorem \ref{morplift} are homotopic. This property will be called by homotopy lifting property. First we remind the notion of homotopy of crossed module morphisms.

\begin{definition}\label{homxmodmor}
	Let $(f_1,g_2),(f_2,g_2)\colon(\widetilde{A},\widetilde{B},\widetilde{\alpha})\rightarrow(A,B,\alpha)$ be two crossed module morphisms. A homotopy from $(f_1,g_1)$ to $(f_2,g_2)$ is a function $d\colon \w{B}\rightarrow A$ satisfying
	\begin{enumerate}[label=\textbf{(H\arabic{*})}, leftmargin=2cm]
		\item $d(\w{b_1}+\w{b_2})=d(\w{b_1})+g_1(\w{b_1})\cdot d(\w{b_2})$,
		\item $d \w{\alpha}(\w{a})=f_1(\w{a})-f_2(\w{a})$ and
		\item $\alpha d(\w{b})=g_1(\w{b})-g_2(\w{b})$
	\end{enumerate}
	for all $\w{b_1},\w{b_2}\in\w{B}$ and $\w{a}\in\w{A}$.
\end{definition}

If $d\colon \w{B}\rightarrow A$ is a homotopy from $(f_1,g_1)$ to $(f_2,g_2)$ then, this is denoted by $d\colon(f_1,g_1)\simeq(f_2,g_2)$.

\begin{theorem}[\textbf{Homotopy Lifting Property}]\label{homliftprop}
	Let $(f_1,g_1),(f_2,g_2) \colon (\widetilde{A},\widetilde{B},\widetilde{\alpha}) \rightarrow (A,B,\alpha)$ be morphisms of crossed modules such that $f_1(\ker \widetilde{\alpha})\subseteq\ker \varphi$, $f_2(\ker \widetilde{\alpha})\subseteq\ker \varphi$ and $(\widetilde{A},\widetilde{B},\widetilde{\alpha})$ is transitive. Let $(\varphi,X,\omega)$ be a lifting of $(A,B,\alpha)$. By Theorem \ref{morplift} there exist crossed module morphisms $(f_1,\widetilde{g_1}),(f_2,\widetilde{g_2})\colon(\widetilde{A},\widetilde{B},\widetilde{\alpha})\rightarrow(A,X,\varphi)$ such that $\omega\widetilde{g_1}=g_1$ and $\omega\widetilde{g_2}=g_2$. In this case, if $d\colon(f_1,g_1)\simeq(f_2,g_2)$, then $d\colon(f_1,\widetilde{g_1})\simeq(f_2,\widetilde{g_2})$.
\end{theorem}

\begin{proof} Let $d\colon \w{B}\rightarrow A$ be a homotopy, i.e., $d\colon (f_1,g_1)\simeq(f_2,g_2)$. Then for all $\w{b_1},\w{b_2}\in\w{B}$ and $\w{a}\in\w{A}$, \textbf{(H1)} $d(\w{b_1}+\w{b_2})=d(\w{b_1})+g_1(\w{b_1})\cdot d(\w{b_2})$, \textbf{(H2)} $d \w{\alpha}(\w{a})=f_1(\w{a})-f_2(\w{a})$ and \textbf{(H3)} $\alpha d(\w{b})=g_1(\w{b})-g_2(\w{b})$. Now we need to show that $d\colon \w{B}\rightarrow A$ is a homotopy from $(f_1,\widetilde{g_1})$ to $(f_2,\widetilde{g_2})$, i.e., $d\colon(f_1,\widetilde{g_1})\simeq(f_2,\widetilde{g_2})$. First two conditions of Definition \ref{homxmodmor} hold, since they are same with base homotopy. Hence, we only need to show that the last condition holds.
	\begin{enumerate}[leftmargin=2cm]
		\item[\textbf{(H3)}] Let $\w{b}\in\w{B}$. Since $\w{\alpha}$ is surjective there exist an element $\w{a}$ such that $\w{\alpha}(\w{a})=\w{b}$. Then,
		\begin{align*}
		\varphi d(\w{b}) &= \varphi(d(\w{\alpha}(\w{a}))) \tag{by \textbf{(H2)}} \\
		&=  \varphi(f_1(\w{a})-f_2(\w{a}))\\
		&= \varphi f_1(\w{a})-\varphi f_2(\w{a}) \tag{by definition of $\w{g_1}$ and $\w{g_2}$}\\
		&= \w{g_1}(\w{b})-\w{g_2}(\w{b}).
		\end{align*}
	\end{enumerate}
	Hence $d\colon \w{B}\rightarrow A$ is a homotopy from $(f_1,\widetilde{g_1})$ to $(f_2,\widetilde{g_2})$.
\end{proof}

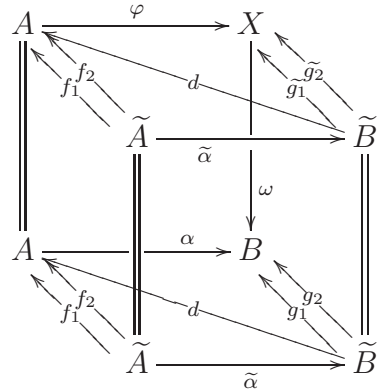
\[\xymatrix@!0@C=15mm@R=15mm{
	A \ar@{=}[dd] \ar[rr]^{\varphi}   & & X \ar@{->}'[d][dd]^-{\omega}  & \\	
	& \w{A} \ar@{=}[dd] \ar[rr]_<<<<<<<{\w{\alpha}} \ar@<0.7ex>[ul]|-{f_1}\ar@<-0.7ex>[ul]|-{f_2}  &  & \w{B} \ar@{=}[dd] \ar@<0.8ex>[ul]|-{\w{g_1}}\ar@<-0.8ex>[ul]|-{\w{g_2}}  \ar@{->}[ulll]|-{d}
	\\
	A \ar@{->}'[r][rr]^-{\alpha}   & & B & \\	
	& \w{A} \ar[rr]_{\w{\alpha}} \ar@<0.7ex>[ul]|-{f_1}\ar@<-0.7ex>[ul]|-{f_2}  &  & \w{B} \ar@<0.7ex>[ul]|-{g_1}\ar@<-0.7ex>[ul]|-{g_2} \ar[ulll]|-{d}
}\]

\subsection{Derivations for the liftings of crossed modules}

The notion of derivations first appears in the work of Whitehead \cite{Wth3} under the name of crossed homomorphisms (see also \cite{KNor}). Let $(A,B,\alpha)$ be a crossed module. All maps $d\colon B\rightarrow A$ such that for all $b,b_1\in B$ \[d(b+b_1)=d(b)+b\cdot d(b_1)\] are called \emph{derivations} of $(A,B,\alpha)$. Any such derivation defines endomorphisms $\theta_d$ and $\sigma_d$ on $A$ and $B$ respectively, as follows:
\[\theta_d(a)=d\alpha(a)+a  \quad \text{and} \quad \sigma_d(b)=\alpha d(b)+b\]

\[\xymatrix{
	A \ar@(ul,dl)[]|{\theta_d} \ar@<-.5ex>[rr]_\alpha
	&& B \ar@/_/[ll]_{d} \ar@(ur,dr)[]|{\sigma_d}}\]

It is easy to see that $(\theta_d,\sigma_d)\colon (A,B,\alpha)\rightarrow (A,B,\alpha)$ is a crossed module morphism such that $d\colon(\theta_d,\sigma_d)\simeq(1_A,1_B)$ and \[\theta_d(d(b))=d(\sigma_d(b))\] for all $b\in B$.


All derivations from $B$ to $A$ are denoted by $\Der(B,A)$. Whitehead defined a multiplication on $\Der(B,A)$ as follows: Let $d_1,d_1\in\Der(B,A)$ then $d=d_1\circ d_2$ where

\[d(b)=d_1\sigma_{d_2}(b)+d_2(b) \ \ \ \ \	(=\theta_{d_1}d_2(b)+d_1(b)).\]

With this multiplication $(\Der(B,A),\circ)$ becomes a semi-group where the identity derivation is zero morphism, i.e. $d\colon B\rightarrow A, b\mapsto d(b)=0$. Furthermore $\theta_d=\theta_{d_1}\theta_{d_2}$ and $\sigma_d=\sigma_{d_1}\sigma_{d_2}$. Group of units of $\Der(B,A)$ is called \emph{Whitehead group} and denoted by $\D(B,A)$. These units are called \emph{regular derivations}. Following proposition is a combined result from \cite{Wth3} and \cite{Lue} to characterize the regular derivations.

\begin{proposition}\cite{KNor}
	Let $(A,B,\alpha)$ be a crossed module. Then the followings are equivalent.
	\begin{enumerate}[label=\textbf{(\roman{*})}, leftmargin=1cm]
		\item $d\in\D(B,A)$;
		\item $\theta\in\Aut A$;
		\item $\sigma\in\Aut B$.
	\end{enumerate}
\end{proposition}

Now we will give some results for derivations of crossed modules in the sense of liftings.

\begin{theorem}\label{derlift}
	Let $(A,B,\alpha)$ be a crossed module, $(\varphi,X,\omega)$ a lifting of $(A,B,\alpha)$ and $d\in\Der(B,A)$. Then $\widetilde{d}=d\omega$ is a derivation of $(A,X,\varphi)$, i.e. $\widetilde{d}\in\Der(X,A)$.
\end{theorem}

\begin{proof}
	Let $x,y\in X$. Then
	\begin{eqnarray*}
		\widetilde{d}(x+y)	&=& d\omega(x+y) \\
		&=& d(\omega(x)+\omega(y)) \\
		&=& d(\omega(x))+\omega(x)\cdot d(\omega(y)) \\
		&=& \widetilde{d}(x)+x\cdot \widetilde{d}(y).
	\end{eqnarray*}
	Thus $\widetilde{d}$ is a derivation of the crossed module $(A,X,\varphi)$, i.e. $\widetilde{d}\in\Der(X,A)$.
\end{proof}

The derivation $\widetilde{d}$ will be called \emph{lifting of $d$ over $\omega$} or briefly \emph{lifting of} $d$. This construction defines a semi-group homomorphism as follows:
\[\begin{array}{ccc}
\Der(B,A) & \longrightarrow & \Der(X,A) \\
d & \longmapsto & \widetilde{d}=d\omega
\end{array}\tag{$\ast$}\label{1}\]
\begin{proposition}\label{endolift}
	Let $(A,B,\alpha)$ be a crossed module, $(\varphi,X,\omega)$ a lifting of $(A,B,\alpha)$, $d\in\Der(B,A)$ and $\widetilde{d}$ be the lifting of $d$. Then $\theta_d=\theta_{\widetilde{d}}$ and $\sigma_d\omega=\omega\sigma_{\widetilde{d}}$.
\end{proposition}

\begin{proof}
	This can be proven by an easy calculation. So proof is omitted.
\end{proof}

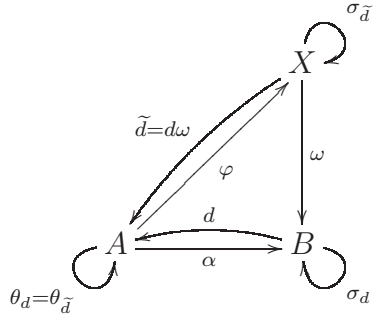
\[\xymatrix{ && X \ar[dd]^\omega \ar@/_/[lldd]_{\widetilde{d}=d\omega} \ar@(u,r)[]^{\sigma_{\widetilde{d}}} \\ && \\ A  \ar@<-.2ex>[uurr]_\varphi \ar@(l,d)[]_{\theta_d=\theta_{\widetilde{d}}} \ar@<-.2ex>[rr]_\alpha
	&& B \ar@/_/[ll]_{d} \ar@(r,d)[]^{\sigma_d}}\]

Combining the results in Theorem \ref{derlift} and Proposition \ref{endolift} we can give the following corollary.

\begin{corollary}\label{dlift}
	Let $(A,B,\alpha)$ be a crossed module, $(\varphi,X,\omega)$ a lifting of $(A,B,\alpha)$ and $d\in\D(B,A)$. Then $\widetilde{d}=d\omega$ is a regular derivation of $(A,X,\varphi)$, i.e. $\widetilde{d}\in\D(X,A)$.
\end{corollary}

This construction defines a group homomorphism as follows:

\[\begin{array}{ccc}
\D(B,A) & \longrightarrow & \D(X,A) \\
d & \longmapsto & \widetilde{d}=d\omega
\end{array}\tag{$\ast\ast$}\label{2}\]

In the following proposition we will obtain a derivation for the base crossed module using a derivation of lifting crossed module.

\begin{proposition}\label{dersink}
	Let $(A,B,\alpha)$ be a crossed module, $(\varphi,X,\omega)$ a lifting of $(A,B,\alpha)$ and $\widetilde{d}\in\Der(X,A)$. Let $\omega$ has a section $s\colon B\rightarrow A$, i.e. $\omega s=1_B$. Then $d=\widetilde{d}s$ is a derivation of $(A,B,\alpha)$, i.e. $d\in\Der(B,A)$.
\end{proposition}

\begin{proof}
	Let $b,b_1\in B$. Then
	\begin{eqnarray*}
		d(b+b_1) &=& \widetilde{d}s(b+b_1) \\
		&=& \widetilde{d}(s(b)+s(b_1)) \\
		&=& \widetilde{d}s(b)+s(b)\cdot\widetilde{d}(s(b_1)) \\
		&=& \widetilde{d}s(b)+\omega s(b)\cdot \widetilde{d}(s(b_1)) \\
		&=& d(b)+b\cdot d(b_1)
	\end{eqnarray*}
	Thus $d$ is a derivation of the base crossed module $(A,B,\alpha)$, i.e. $d\in\Der(B,A)$.
\end{proof}

In this case the (\ref{1}) homomorphism becomes injective. Thus $\Der(B,A)$ can be consider as a semi-subgroup of $\Der(X,A)$.

%

\begin{corollary}
	Let $(A,B,\alpha)$ be a crossed module, $(\varphi,X,\omega)$ a lifting of $(A,B,\alpha)$ and $\widetilde{d}\in\D(X,A)$. Let $\omega$ has a section $s\colon B\rightarrow A$, i.e. $\omega s=1_B$. Then $d=\widetilde{d}s$ is a regular derivation of $(A,B,\alpha)$, i.e. $d\in\D(B,A)$.
\end{corollary}

Under the conditions of this corollary the (\ref{2}) homomorphism becomes injective. Thus $\D(B,A)$ can be consider as a subgroup of $\D(X,A)$.

\small{

}
\end{document}